\documentclass[12pt,fleqn]{article}

\usepackage{amsmath}
\usepackage{amsthm}
\usepackage{amssymb}
\usepackage{amsfonts}
\usepackage{graphicx}
\usepackage{color}
\newtheorem{theorem}{Theorem}[section]
\newtheorem{lemma}[theorem]{Lemma}
\newtheorem{corollary}[theorem]{Corollary}

\theoremstyle{definition}
\newtheorem{definition}[theorem]{Definition}

\newtheorem{criterion}[]{Criterion}
\newtheoremstyle{named}{}{}{\itshape}{}{\bfseries}{.}{.5em}{\thmnote{#3's }#1} \theoremstyle{named} 

\theoremstyle{remark}

\numberwithin{equation}{section}

\numberwithin{equation}{section}

\title{On the minimum distance between masses of relative equilibria of the $n$-body problem}
\author{Pieter Tibboel$^\ast$}

\begin{document}
\maketitle
\begin{abstract}
  We prove that if for relative equilibrium solutions of a generalisation of the $n$-body problem of celestial mechanics the masses and rotation are given, then the minimum distance between the point masses of such a relative equilibrium has a universal lower bound that is not equal to zero. We furthermore prove that the set of such relative equilibria is compact.
\end{abstract}

\begin{description}

\item \hspace*{3.8mm}$\ast$ Department of Mathematics, City University of
Hong Kong, Hong Kong. \\
Email: \texttt{ptibboel@cityu.edu.hk}

\end{description}

\newpage
\section{Introduction}

By the $n$-body problem we mean the problem of deducing the dynamics of $n$ point masses with time dependent coordinates $q_{1}$,..., $q_{n}\in\mathbb{R}^{k}$, $k\geq 2$ and respective masses $m_{1}$,...,$m_{n}$ as described by the system of differential equations \begin{align}\label{Equations of motion}
  \ddot{q}_{i}=\sum\limits_{j=1,\textrm{ }j\neq i}^{n}m_{j}(q_{j}-q_{i})\|q_{j}-q_{i}\|^{2a},\textrm{ }n\geq 3,\textrm{ }a<-\frac{1}{2}.
\end{align}
If $a=-\frac{3}{2}$ and $k=3$, then we speak of the \textit{classical $n$-body problem}.
We call any solution to such a problem where the $q_{1}$,..., $q_{n}$ describe a rotating configuration of points a \textit{relative equilibrium} and the set of all such configurations that are equivalent under rotation and scalar multiplication a \textit{class of relative equilibria}.

Steve Smale conjectured on his famous list (see \cite{Smale}) after Wintner (see \cite{Wintner}) that for the classical case, if the equilibria are induced by a plane rotation, the number of classes of relative equilibria is finite, if the masses $m_{1}$,...,$m_{n}$ are given. This problem is still open for $n>5$ and was solved for $n=3$ by A. Wintner (see \cite{Wintner}), $n=4$ by M. Hampton and R. Moeckel (see \cite{HamptonMoeckel}) and for $n=5$ by A. Albouy and V. Kaloshin (see \cite{AlbouyKaloshin}). \\
Results on the finiteness of subclasses of relative equilibria can be found in \cite{Kuzmina}, \cite{Llibre}, \cite{Moulton} and \cite{Palmore}.
G. Roberts showed in \cite{Roberts} that for the classical five-body problem, if one of the masses is negative, a continuum of relative equilibria exists. As a potential step towards a proof of Smale's problem, M. Shub showed in \cite{S} that the set of all classes of relative equilibria, provided they have the same set of masses, is compact. Moreover, Shub proved, again in \cite{S}, that if the rotation inducing the equilibria is given as well, then there exists a universal nonzero, minimal distance that the point masses lie apart from each other. For further background information and a more detailed overview regarding Smale's problem, see \cite{AbrahamMarsden},  \cite{HamptonMoeckel}, \cite{Smale2} and \cite{Roberts} and the references therein. \\
In this paper, as a logical next step after Shub's work in \cite{S}, we prove Shub's results when using (\ref{Equations of motion}) instead of the classical $n$-body problem. Specifically, we prove that
\begin{theorem}\label{limits distinct}
  Consider the set $R_{A,m_{1},...,m_{n}}$ of all relative equilibria with rotation matrix  $T_{k}(\overrightarrow{A}t)$ and masses $m_{1}$,..., $m_{n}$ (see Definition~\ref{Definition Relative Equilibrium}). Then there exists a constant $c\in\mathbb{R}_{>0}$ such that for all relative equilibria $\{T_{k}(\overrightarrow{A}t)Q_{i}\}_{i=1}^{n}$ in the set $R_{A,m_{1},...,m_{n}}$, we have that $\|Q_{i}-Q_{j}\|>c$ for all $i$, $j\in\{1,...,n\}$, $i\neq j$.
\end{theorem}
and consequently that
\begin{corollary}\label{compact}
  Consider the set $R_{A,m_{1},...,m_{n}}$ of all relative equilibria with rotation matrix  $T_{k}(\overrightarrow{A}t)$ and masses $m_{1}$,..., $m_{n}$ (see Definition~\ref{Definition Relative Equilibrium}). Then there exists a $C\in\mathbb{R}_{>0}$ such that for all relative equilibria $\{T_{k}(\overrightarrow{A}t)Q_{i}\}_{i=1}^{n}$ in the set $R_{A,m_{1},...,m_{n}}$, we have that $\|Q_{i}\|<C$ for all $i\in\{1,...,n\}$.
\end{corollary}
In order to prove Theorem~\ref{limits distinct} and Corollary~\ref{compact}, we will first formulate needed definitions, a criterion for relative equilibria and a lemma related to relative equilibria, which will be done in section~\ref{Background Theory}. Then we will prove Theorem~\ref{limits distinct} in section~\ref{Proof of the Main Theorem} and Corollary~\ref{compact} in section~\ref{Proof of the Main Corollary}. %This paper is inspired by \cite{D2} and \cite{S}.

\section{Background Theory}\label{Background Theory}
Before getting to the lemma that will form the backbone of our theorems, we will have to formulate a criterion, for which we will have to adopt the following definition:\\
  Let
  \begin{align*}
    T(t)=\begin{pmatrix}
      \cos{t} &-\sin{t} \\
      \sin{t} & \cos{t}
    \end{pmatrix}
  \end{align*}
  and define for any $p$-dimensional vector-valued function \begin{align*}\overrightarrow{\theta}=\begin{pmatrix}\theta_{1}\\ \vdots\\ \theta_{p}\end{pmatrix}
  \end{align*}
  a $k\times k$ diagonal block matrix $T_{k}(\overrightarrow{\theta})$ as
  \begin{align*}
    T_{k}(\overrightarrow{\theta})=\begin{cases}
      \begin{pmatrix}
        T(\theta_{1}) & \hdots & 0 \\
        \vdots & \ddots & \vdots \\
        0 & \hdots & T(\theta_{p})
      \end{pmatrix}\textrm{ if }k=2p, \\
      \begin{pmatrix}
        T(\theta_{1}) & \hdots & 0 & 0\\
        \vdots & \ddots & \vdots &\vdots \\
        0 & \hdots & T(\theta_{p}) & 0 \\
        0 & \hdots & 0 & 1
      \end{pmatrix}\textrm{ if }k=2p+1 \\
    \end{cases}.
  \end{align*}
  Then
\begin{definition}\label{Definition Relative Equilibrium}
  Let $n\in\mathbb{N}$, let $q_{i}$, $i\in\{1,...,n\}$ solve (\ref{Equations of motion}), let \begin{align*}
    \overrightarrow{A}=\begin{pmatrix}A_{1}\\ \vdots \\A_{p}\end{pmatrix}\in\mathbb{R}_{>0}^{p},
  \end{align*} $Q_{1}$,...,$Q_{n}\in\mathbb{R}^{k}$ and let $q_{i}(t)=T_{k}(\overrightarrow{A}t)Q_{i}$, $i=1,...,n$. Then we say that $q_{1}$,..., $q_{n}$ form a \textit{relative equilibrium} with rotation matrix $T_{k}(\overrightarrow{A}t)$.
\end{definition}
Inserting the expressions for $q_{1}$,..., $q_{n}$ as described in Definition~\ref{Definition Relative Equilibrium} into (\ref{Equations of motion}) and using that for any $x\in\mathbb{R}^{k}$, $\|T_{k}(\overrightarrow{A}t)x\|=\|x\|$ and that $(T_{k}(\overrightarrow{A}t))''=-\mathbf{A}^{2}T_{k}(\overrightarrow{A}t)$ where $\mathbf{A}$ is the diagonal matrix
 \begin{align*}
   \mathbf{A}=\begin{cases}
     \begin{pmatrix}
       A_{1} & 0  & \hdots & 0 & 0\\
       0 & A_{1} & \ddots & \vdots & \vdots \\
       \vdots & \ddots & \ddots & 0 & 0\\
       0 & \hdots & 0 & A_{p} & 0\\
       0 & \hdots & 0 & 0 & A_{p}
     \end{pmatrix}\textrm{ if }k=2p\\
     \begin{pmatrix}
       A_{1} & 0  & \hdots & 0 & 0 & 0\\
       0 & A_{1} & \ddots & \vdots & \vdots& \vdots \\
       \vdots & \ddots & \ddots & 0 & 0 & 0\\
       0 & \hdots & 0 & A_{p} & 0 & 0 \\
       0 & \hdots & 0 & 0 & A_{p} & 0 \\
       0 & \hdots & 0 & 0 & 0 & 0
     \end{pmatrix}\textrm{ if }k=2p+1
   \end{cases}
 \end{align*} we get
\begin{criterion}\label{Criterion}
  For any relative equilibrium solution $\{T_{k}(\overrightarrow{A}t)Q_{i}\}_{i=1}^{n}$ of (\ref{Equations of motion}) as described in Definition~\ref{Definition Relative Equilibrium}, we have that
  \begin{align*}
    \mathbf{A}^{2}Q_{i}=\sum\limits_{j=1,\textrm{ }j\neq i}^{n}m_{j}(Q_{i}-Q_{j})\|Q_{i}-Q_{j}\|^{2a},\textrm{ }i\in\{1,...,n\}.
  \end{align*}
\end{criterion}
Before being able to prove Theorem~\ref{limits distinct} and Corollary~\ref{compact}, we will need the following lemma:
\begin{lemma}\label{Lemma Winning Identity}
  Let $q_{1}=T_{k}(\overrightarrow{A}t)Q_{1},...,q_{n}=T_{k}(\overrightarrow{A}t)Q_{n}$ be a relative equilibrium according to Definition~\ref{Definition Relative Equilibrium}. If for $i$, $l\in\{1,...,n\}$ we write
  \begin{align*}R_{il}=\sum\limits_{j=l+1}^{n}m_{j}(Q_{i}-Q_{j})\|Q_{i}-Q_{j}\|^{2a},\end{align*}
  then
  \begin{align}\label{bla1}
    \mathbf{A}^{2}\sum\limits_{i=2}^{l}m_{i}(Q_{1}-Q_{i})=\left(\sum\limits_{i=1}^{l}m_{i}\right)\sum\limits_{j=2}^{l}m_{j}(Q_{1}-Q_{j})\|Q_{1}-Q_{j}\|^{2a}+\sum\limits_{i=2}^{l}m_{i}(R_{1l}-R_{il}),
  \end{align}
\end{lemma}
\begin{proof}
  For $i\in\{2,...,l\}$, using Criterion~\ref{Criterion}, we get
  \begin{align}
    \mathbf{A}^{2}(Q_{1}-Q_{i})&=\sum\limits_{j=2}^{l}m_{j}(Q_{1}-Q_{j})\|Q_{1}-Q_{j}\|^{2a}-\sum\limits_{j=1,\textrm{ }j\neq i}^{l}m_{j}(Q_{i}-Q_{j})\|Q_{i}-Q_{j}\|^{2a}+R_{1l}-R_{il}\nonumber \\
    &=\sum\limits_{j=2}^{l}m_{j}(Q_{1}-Q_{j})\|Q_{1}-Q_{j}\|^{2a}-\sum\limits_{j=2,\textrm{ }j\neq i}^{l}m_{j}(Q_{i}-Q_{j})\|Q_{i}-Q_{j}\|^{2a}\nonumber\\
    &+m_{1}(Q_{1}-Q_{i})\|Q_{1}-Q_{i}\|^{2a}+R_{1l}-R_{il}.\label{bla}
  \end{align}
  Note that
  \begin{align*}
    \sum\limits_{i=2}^{l}\sum\limits_{j=2,\textrm{ }j\neq i}^{l}m_{i}m_{j}(Q_{i}-Q_{j})\|Q_{i}-Q_{j}\|^{2a}=0,
  \end{align*}
  so multiplying both sides of (\ref{bla}) with $m_{i}$ and then summing both sides over $i$ from $2$ to $l$ gives
  \begin{align}
    \mathbf{A}^{2}\sum\limits_{i=2}^{l}m_{i}(Q_{1}-Q_{i})&=\left(\sum\limits_{i=2}^{l}m_{i}\right)\sum\limits_{j=2}^{l}m_{j}(Q_{1}-Q_{j})\|Q_{1}-Q_{j}\|^{2a}-0\\
    &+\sum\limits_{i=2}^{l}m_{1}m_{i}(Q_{1}-Q_{i})\|Q_{1}-Q_{i}\|^{2a}+\sum\limits_{i=2}^{l}m_{i}(R_{1l}-R_{il})\nonumber\\
    &=\left(\sum\limits_{i=1}^{l}m_{i}\right)\sum\limits_{j=2}^{l}m_{j}(Q_{1}-Q_{j})\|Q_{1}-Q_{j}\|^{2a}+\sum\limits_{i=2}^{l}m_{i}(R_{1l}-R_{il}).\label{limits distinct 2}
  \end{align}
\end{proof}
We now have all that is needed to prove our main theorem.
\section{Proof of Theorem~\ref{limits distinct}}\label{Proof of the Main Theorem}
\begin{proof}
  Assume that the contrary is true. Then there exist sequences $\{Q_{ir}\}_{r=1}^{\infty}$ and relative equilibria $q_{ir}(t)=T_{k}(\overrightarrow{A}t)Q_{ir}$, $i\in\{1,...,n\}$ for which we may assume, if we renumber the $Q_{ir}$ in terms of $i$ and take subsequences if necessary, the following:
  \begin{itemize}
    \item[1. ] There exist sequences $\{Q_{1r}\}_{r=1}^{\infty}$,...,$\{Q_{lr}\}_{r=1}^{\infty}$, $l\leq n$ such that $\|Q_{ir}-Q_{jr}\|$ goes to zero for $r$ going to infinity if $i$, $j\in\{1,...,l\}$.
    \item[2. ] $\|Q_{ir}-Q_{jr}\|$ does not go to zero for $r$ going to infinity if $i\in\{1,...,l\}$ and $j\in\{l+1,...,n\}$.
    \item[3. ] $Q_{1r}$,...,$Q_{lr}$ do not go to zero, as any solution of (\ref{bla1}) is determined up to rotation and translation, so by translating  $Q_{1r}$,...,$Q_{lr}$ if necessary, we may assume that $Q_{1r}$,...,$Q_{lr}$ do not go to zero.
    \item[4. ] $\|Q_{1r}-Q_{lr}\|\geq\|Q_{ir}-Q_{jr}\|$ for all $i$, $j\in\{1,...,l\}$, for all $r\in\mathbb{N}$.
  \end{itemize}
  Note that for any $i\in\{1,...,l\}$ the vectors $Q_{1r}-Q_{lr}$, $Q_{1r}-Q_{ir}$ and $Q_{ir}-Q_{lr}$ either form a triangle with $\|Q_{1r}-Q_{lr}\|$ the length of its longest side, or the three of them align, meaning the angles between them are zero. Consequently, the angle between $Q_{1r}-Q_{lr}$ and $Q_{1r}-Q_{ir}$ is smaller than $\frac{\pi}{2}$. Let $\beta_{i1lr}$ be the angle between $Q_{1r}-Q_{lr}$ and $Q_{1r}-Q_{ir}$. If there are $i$ such that $\lim\limits_{r\rightarrow\infty}\beta_{i1lr}<\frac{1}{2}\pi$, then taking inner products on both sides of (\ref{limits distinct 2}) with $\frac{Q_{1r}-Q_{lr}}{\|Q_{1r}-Q_{lr}\|}$ and then letting $r$ go to infinity gives a contradiction. As $\lim\limits_{r\rightarrow\infty}\beta_{l1lr}=0$, there is at least one such an $i$. This completes the proof.
\end{proof}
\section{Proof of Corollary~\ref{compact}}~\label{Proof of the Main Corollary}
\begin{proof}
  Assume the contrary to be true. Then there exist sequences $\{Q_{ir}\}_{r=1}^{\infty}$, $i\in\{1,...,n\}$ for which $q_{ir}(t)=T_{k}(\overrightarrow{A}t)Q_{ir}$ define relative equilibrium solutions of (\ref{Equations of motion}) and for which there has to be at least one sequence $\{Q_{ir}\}_{r=1}^{\infty}$ that is unbounded. Taking subsequences and renumbering the $Q_{ir}$ in terms of $i$ if necessary, we may assume that $\{Q_{1r}\}_{r=1}^{\infty}$ is unbounded. By Criterion~\ref{Criterion},
  \begin{align}\label{Compactness k}
    \mathbf{A}^{2}Q_{1r}=\sum\limits_{j=2}^{n}m_{j}(Q_{1r}-Q_{jr})\|Q_{1r}-Q_{jr}\|^{2a}.
  \end{align}
  As the left-hand side of (\ref{Compactness k}) is unbounded, the right-hand side must be unbounded as well, which means that there must be $j\in\{2,...,n\}$ for which \begin{align*}m_{j}(Q_{1r}-Q_{jr})\|Q_{1r}-Q_{jr}\|^{2a}\end{align*} is unbounded if we let $r$ go to infinity. But as \begin{align*}\left\|m_{j}(Q_{1r}-Q_{jr})\|Q_{1r}-Q_{jr}\|^{2a}\right\|=m_{j}\|Q_{1r}-Q_{jr}\|^{2a+1},\end{align*} that means that $\|Q_{1r}-Q_{jr}\|$ goes to zero for $r$ going to infinity, which is impossible by Theorem~\ref{limits distinct}. This completes the proof.
\end{proof}

\end{document}